\documentclass[11pt]{amsart}
\usepackage{amssymb,amsmath,amsthm,amsfonts,amsopn,url,color,hyperref,enumitem}
\usepackage[mathscr]{euscript}
\usepackage[all]{xy}

\theoremstyle{plain}
\newtheorem{thm}{Theorem}
\newtheorem*{mainthm}{Main Theorem}
\newtheorem{prop}[thm]{Proposition}
\newtheorem{lemma}[thm]{Lemma}

\theoremstyle{definition}

\newtheorem*{defn*}{Definition}
\newtheorem*{question*}{Question}

\newtheorem{example}[thm]{Example}
\newtheorem*{example*}{Example}
\newtheorem{rem}[thm]{Remark}
\newtheorem*{rem*}{Remark}
\newcommand{\field}[1]{\mathbb{#1}}
\newcommand{\N}{\field{N}}

\newcommand{\ideal}[1]{\mathfrak{#1}}
\newcommand{\m}{\ideal{m}}

\newcommand{\ia}{\ideal{a}}

\newcommand{\ra}{\rightarrow}

\DeclareMathOperator{\ann}{ann}

\newcommand{\be}{\begin{enumerate}}
\newcommand{\ee}{\end{enumerate}}

\newcommand{\li}
 {\leftfootline}

\renewcommand{\phi}{\varphi}

\author{Neil Epstein}
\address{Department of Mathematical Sciences \\ George Mason University \\ Fairfax, VA  22030}
\email{nepstei2@gmu.edu}
\date{February 21, 2021}
\title{Regularity and intersections of bracket powers}

\begin{document}
\begin{abstract}
    Among reduced Noetherian prime characteristic  commutative rings, we prove that a regular ring is precisely one where finite intersection of ideals commutes with taking bracket powers. However, reducedness is essential for this equivalence. Connections are made with Ohm-Rush content theory, intersection-flatness of the Frobenius map, and various flatness criteria.  
\end{abstract}

\maketitle

In this note, all rings are unital, associative, and  (unless otherwise specified) commutative.

The condition of being \emph{regular} is paramount in the theory of commutative Noetherian rings, due in part to its connections to algebraic geometry.  Hence an easily tested algebraic condition for it is always welcome.  In pursuance of this, we offer:

\begin{mainthm}\label{mainthm}
Let $R$ be a Noetherian reduced commutative ring of prime characteristic $p>0$.  Then the following are equivalent: \begin{enumerate}
 \item\label{it:mainreg} $R$ is regular.
 \item\label{it:mainwOR} Finite intersection of ideals commutes with taking bracket powers.
 \item\label{it:captwo} For any ideal $I$ and any $x\in R$, we have $I^{[p]} \cap (x^p) = (I \cap (x))^{[p]}$.
 \item\label{it:colon} For any ideal $I$ and any $x\in R$, we have $(I :x)^{[p]} = (I^{[p]} : x^p)$.
\end{enumerate}
\end{mainthm}

The use of the Frobenius to detect important properties of prime characteristic rings, including regularity, the complete intersection property, and Gorensteinnness, is not new.  We refer the reader to Miller's survey \cite{Mil-Frobhd}. 
Indeed, the equivalence of (\ref{it:mainreg}) and (\ref{it:colon}) of the above theorem was previously proved by Zhang \cite{Zh-Frobcolon}, using completely different methods from the ones employed here.

Recall that for a power $p^n$ of $p$, the $p^n$th \emph{bracket power} of an ideal $I$ is the ideal generated by $\{x^{p^n} : x\in I\}$.  Kunz \cite{Kunz-regp} showed\footnote{Kunz only stated this for $R$ reduced, but as explained in tbe footnote to \cite[Theorem 3.6]{nmeYao-flat}, this is an unnecessary assumption when the result is stated this way.} that for a commutative Noetherian ring $R$ of prime characteristic $p>0$, $R$ is regular if and only if the Frobenius endomorphism $R \rightarrow R$, $r \mapsto r^p$, is \emph{flat}.  Indeed, the proof of our main theorem is really about flatness, so we commence with the \emph{Bourbaki flatness criterion}, a criterion that this author feels should be better known than it is.

\begin{prop}\label{prop:Bourflat}\cite[Exercise I.2.22]{Bour-CA}
Let $R$ be a ring, not necessarily commutative, and let $M$ be a left $R$-module.  Then $M$ is $R$-flat if and only if $(IM :_M x) = (I :_R x)M$ for all finitely generated right ideals $I$ and every $x\in R$, where for a subset $L \subseteq M$, $(L :_M x) := \{z \in M : xz \in L\}$.
\end{prop}

Next, we have the \emph{Hochster-Jeffries flatness criterion}.

\begin{thm}\label{thm:HJflat}
Let $R$ be a ring, not necessarily commutative, and $M$ a left $R$-module.  Then $M$ is flat over $R$ if and only if \begin{enumerate}
\item\label{HJcap} For any finitely generated right ideal $I$ of $R$ and any $x\in R$, we have $(IM \cap xM) = (I \cap xR)M$, and
\item\label{HJann} For any $x\in R$, we have $(0 :_M x) = (0 :_R x)M$.
\end{enumerate}
\end{thm}

\begin{rem}
The above is proved by Hochster and Jeffries in \cite[Proposition 5.5, iv $\implies$ i]{HoJe-flatness} in case $R$ is commutative, though this is not how they state it.  Bourbaki flatness makes the proof simpler though, and it allows the passage to noncommutative rings.
\end{rem}

\begin{proof}[Proof of Theorem~\ref{thm:HJflat}]
First suppose $M$ is flat over $R$.  Then (\ref{HJcap}) is well known (by applying $- \otimes_R M$ to the exact sequence of right $R$-modules $0 \ra R/(I \cap xR) \ra (R/I) \oplus (R/xR)$), and (\ref{HJann}) follows from Bourbaki flatness (Proposition~\ref{prop:Bourflat}).

Conversely, suppose (\ref{HJcap}) and (\ref{HJann}) hold.  Let $I$ be a finitely generated right ideal and $x\in R$.  Then we have \[
x(IM :_M x) = IM \cap xM = (I \cap xR)M = x(I :_R x) M.
\]
It follows that \[
IM :_M x = (I :_R x)M + (0 :_M x) = (I :_R x) M + (0 :_R x) M = (I :_R x) M.
\]
Another appeal to Bourbaki flatness finishes the proof.
\end{proof}

\begin{lemma}\label{lem:domain}
Let $R$ be a (commutative) integral domain and $M$ a torsion-free $R$-module such that for any finitely generated ideal $I$ of $R$ and any $x\in R$, we have $(I \cap xR)M = IM \cap xM$.  Then $M$ is flat over $R$.
\end{lemma}

\begin{proof}
By Hochster-Jeffries flatness (Theorem~\ref{thm:HJflat}), it is enough to show that $\ann_R(a)M = (\ann_M a)$ for all $a\in R$.  But if $a=0$, then $\ann_R(a) = R$ and $\ann_M(a) = M$, whence $\ann_R(a)M = RM = M = \ann_Ma$.  On the other hand, if $a\neq 0$, then $\ann_R(a)=0$, and by torsion-freeness we also have $\ann_M(a) = 0$, so $\ann_R(a)M = 0M = 0 = \ann_M(a)$.
\end{proof}

Next, note the following connection between zero-divisors and intersection of bracket powers of principal ideals.

\begin{lemma}\label{lem:zd}
Let $(R,\m)$ be a commutative Noetherian local ring of prime characteristic $p>0$.  Suppose that for all $x,y \in R$, we have $((x^p) \cap (y^p)) = ((x) \cap (y))^{[p]}$.  Then every nonzero zero-divisor in $R$ has a nonzero nilpotent multiple.
\end{lemma}

\begin{proof}
Let $0 \neq a$ be a zero-divisor of $R$.  Assume it has no nonzero nilpotent multiple.  Let $P$ be an associated prime of $R$ that contains $\ann(a)$.  Without loss of generality (by passing to a multiple), we may assume $P=\ann(a)$.  Now let $Z := \{\ann x : 0 \neq x \in P\}$.  Then $Z$ is a nonempty set of proper ideals of $R$, so the fact that $R$ is Noetherian implies that $Z$ has a maximal element $Q$, which is then prime by the usual arguments.  Let $x\in P$ with $Q=\ann x$.  Then we have $P = \ann(a^t)$ (since $a$ is non-nilpotent) and $Q = \ann(x^t)$ (since $Q$ is maximal in $Z$) for all positive integers $t$.

Fix a positive integer $t$, and let $b=a^t$ and $y=x^t$.  Now let $g\in (b) \cap (b+y)$.  We have $g=cb=d\cdot (b+y)$ for some $c,d \in R$.  Thus, $(c-d)b = dy$, so $(c-d)b^2 = dby = 0$, whence $c-d \in \ann(b^2)=P = \ann(b)$.  Say $c=d+\pi$, $\pi \in P$.  Then we have $0=\pi b = (c-d)b = dy$, so $d \in \ann(y) = Q$.  We have shown that $(b) \cap (b+y) \subseteq Qb$.  Conversely, if $q\in Q$, then $qb = q(b+y) \in (b) \cap (b+y)$ since $qy=0$.  Thus, we have $(b) \cap (b+y) = Qb$.  In particular, choosing $t=1$ yields $(a) \cap (a+x) = Qa$, whence $(a^p) \cap ((a+x)^p) = ((a) \cap (a+x))^{[p]} = (Qa)^{[p]} = Q^{[p]} a^p$.  On the other hand, choosing $t=p$ yields $(a^p) \cap ((a+x)^p) = (a^p) \cap (a^p +x^p) = Qa^p$.  Thus, $Qa^p = Q^{[p]} a^p \subseteq \m Q a^p$, so that by the Nakayama lemma, $Qa^p = 0$.  But $a\in Q$, so $a^{p+1} = 0$, whence $a$ is nilpotent, which is a contradiction.
\end{proof}

We are ready to prove the main theorem.

\begin{proof}[Proof of Main Theorem]
By the Kunz regularity criterion \cite{Kunz-regp}, $R$ is regular iff the maps $f_q: R \rightarrow R$ given by $a \mapsto a^q$, $q$ a power of $p$, are flat, iff $f_p$ is flat.    Moreover, we have $f_q(I)R = I^{[q]}$.  Hence, the equivalence of (1) and (4) follows directly from Bourbaki flatness. Moreover,  the implication (\ref{it:mainreg}) $\implies$ (\ref{it:mainwOR}) follows from the well-known fact that for any flat $R$-module $M$, we have $IM \cap JM = (I \cap J) M$ for any ideals $I$, $J$ of $R$ (cf. the proof of Theorem~\ref{thm:HJflat}), plus a trivial induction step.  The implication (\ref{it:mainwOR}) $\implies$ (\ref{it:captwo}) is obvious.  Hence, we need only prove (\ref{it:captwo}) $\implies$ (\ref{it:mainreg}).

Assume $R$ satisfies (\ref{it:captwo}), and suppose the implication (\ref{it:captwo}) $\implies$ (\ref{it:mainreg}) holds when $R$ is local.  Since both finite intersections and bracket powers commute with localization at maximal ideals
, and since any localization of a reduced ring is reduced, we have that (\ref{it:captwo}) holds for the reduced local ring $R_\m$ for any maximal ideal $\m$.  Thus by assumption, $R_\m$ is a regular local ring, so $R$ is regular since $\m$ was arbitrary.  Thus, we are reduced to the local case.

If $R$ is an integral domain and (\ref{it:captwo}) holds, then since $R^{1/p}$ is a torsion-free $R$-module, the Frobenius is flat by Lemma~\ref{lem:domain}, and we are done by the Kunz regularity criterion.

It remains to show that when $(R,\m)$ is local and (\ref{it:captwo}) holds, $R$ is an integral domain. But this follows from Lemma~\ref{lem:zd} since $R$ is reduced.
\end{proof}

The obvious question now is: Do the equivalences in the Main Theorem still hold when $R$ is nonreduced?  The answer is no.

\begin{example}\label{ex:dualnumbers}
Let $R = k[x]/(x^2)$, where $k$ is a field and $x$ an indeterminate over $k$. Clearly $R$ is local and Noetherian but not regular, since $\dim R=0$ but the maximal ideal requires a nonzero generator.  However, let $I$, $J$ be nonzero ideals of $R$.  Since the 3 ideals of $R$ are linearly ordered, without loss of generality we have $J \subseteq I$.  Thus, $I^{[q]} \cap J^{[q]} = J^{[q]} = (I \cap J)^{[q]}$.
\end{example}

\begin{rem}[Ohm-Rush content and Frobenius roots]
Recall that for a commutative ring $R$, a module $M$ is \emph{Ohm-Rush} \cite{OhmRu-content, nmeSh-OR} or \emph{weakly intersection flat for ideals} \cite{HoJe-flatness}, if for all collections $\{I_\alpha\}$ of ideals of $R$, we have $\bigcap_\alpha (I_\alpha M) = (\bigcap_\alpha I_\alpha)M$.  We prefer the former term as it was Ohm and Rush who first explored this property, albeit using problematic terminology.  The condition on a Noetherian local ring of positive prime characteristic that the Frobenius itself is a flat Ohm-Rush module is an important condition in tight closure theory (cf. \cite{HHbase, Ka-param, Sha-nonred}).  Indeed, we investigate the stronger property of \emph{F-intersection flatness} in \cite{nme-exctest}, proving e.g. that if $A$ is a regular local ring of prime characteristic and $W$ is a multiplicative subset of $B=A[x_1, \ldots, x_n]$, then $W^{-1}B$ is F-intersection flat.

Given an ideal $J$ and a power $q$ of $p$, the Ohm-Rush property of the Frobenius allows one to find a \emph{unique smallest} ideal $I$ such that $J \subseteq I^{[q]}$, which is known as $J^{[1/q]}$ (e.g. see \cite[Definition 9.5]{Sha-nonred}).  In particular, we intersect all the ideals that have this property.  But this intersection is just the Ohm-Rush content of the ideal \cite{OhmRu-content, nmeSh-OR} via the iterated Frobenius endomorphism, so the Ohm-Rush property itself is enough to guarantee such a minimal member.  Hence it is natural to consider (unlike in \cite{nme-exctest}, since F-intersection flatness already implies flatness) whether the weaker condition of being an Ohm-Rush module is enough to force flatness of the Frobenius, and hence regularity.
We see above that the answer is \emph{yes} for reduced rings, but \emph{no} otherwise.  
It might be interesting to characterize the class of (nonreduced) rings for which the Ohm-Rush property holds for the Frobenius endomorphism, of which the above example is a member.  This will have no bearing on tight closure theory in particular (since one can compute tight closure modulo minimal primes), but it is of fundamental interest in prime characteristic algebra more generally.
\end{rem}

When $R$ is complete, finite intersection of ideals commuting with bracket powers is enough to ensure the property for arbitrary intersections of ideals (compare \cite[Proposition 5.7(e)]{HoJe-flatness} and \cite[Proposition 5.3]{Ka-param}, where  flatness of the Frobenius is assumed):

\begin{prop}
Let $R$ be a complete Noetherian local ring of prime characteristic $p>0$. Suppose that for any ideals $I, J$ of $R$, we have $I^{[p]} \cap J^{[p]} = (I \cap J)^{[p]}$.  Then the Frobenius endomorphism $f: R \ra R$ is Ohm-Rush.
\end{prop}

\begin{proof}
We need to show that for any $x\in R$, there is a unique smallest ideal $I$ with $x \in I^{[p]}$. 

Let $\{I_\alpha\}_{\alpha \in \Lambda}$ be the set of ideals $J$ with $x \in J^{[p]}$, indexed by the index set $\Lambda$.  Set $c(x) := \bigcap \{I_\alpha : \alpha \in \Lambda\}$. By the argument in the proof of \cite[Proposition 5.3]{Ka-param}, there is a countable subset $\{\alpha_i : i \in \N\}$ of $\Lambda$ such that $c(x) = \bigcap_{i=1}^\infty I_{\alpha_i}$.  For each $i\in \N$, set $J_i := \bigcap_{h=1}^i I_{\alpha_h}$.  Then this is a decreasing sequence of ideals whose intersection is $c(x)$.  Moreover, for each $i$, we have $x \in \bigcap_{h=1}^i I_{\alpha_h}^{[p]} = (\bigcap_{h=1}^i I_{\alpha_h})^{[p]} = J_i^{[p]}$, where the first equality follows from the assumption on pairs of ideals.  Chevalley's lemma \cite[Lemma 7]{Che-local} then guarantees that for each $n\in \N$, there is some $i(n) \in \N$ such that $J_{i(n)} \subseteq \m^n + c(x)$.  Thus, for each $n$, we have \[
x \in \bigcap_n (J_{i(n)}^{[p]}) \subseteq \bigcap_n (\m^n + c(x))^{[p]} = \bigcap_n (\m^{[p]})^n + c(x)^{[p]} = c(x)^{[p]},
\]
with the last equality by the Krull intersection theorem.  Since $x \in c(x)^{[p]}$ and $c(x)$ is contained in all ideals $J$ with $x \in J^{[p]}$, it follows that $c(x)$ is the unique smallest such ideal.
\end{proof}

Recall that the \emph{Frobenius closure} $I^F$ of an ideal $I$ in a prime characteristic ring consists of all those elements $x\in R$ such that there exists some $n\in \N$ with $x^{p^n} \in I^{[p^n]}$.  Recall that an \emph{F-pure} ring has the property that every ideal is Frobenius closed (and is characterized by the property that every \emph{submodule} of every finite module is Frobenius-closed).  F-purity is considered to be a much weaker property than regularity.
.

Example~\ref{ex:dualnumbers} indicates that there might be a connection between the properties in the Main Theorem for non-reduced rings and the regularity property for the reduced structure of the ring.  Indeed this is so, in the presence of F-purity.

\begin{prop}
Let $R$ be a commutative Noetherian ring of prime characteristic $p>0$ such that \begin{enumerate}
    \item $(I \cap xR)^{[p]} = I^{[p]} \cap x^pR$ for all ideals $I$ and $x\in R$, and 
    \item $R_{red}$ is F-pure.
\end{enumerate}
Then $R_{red}$ is regular.
\end{prop}

\begin{proof}
Let $\ia$ be an ideal of $R_{red}$ and $\alpha \in R_{red}$.  Let $I$ be an ideal of $R$ and $x \in R$ whose residues in $R_{red}$ are $\ia$ and $\alpha$ respectively.  Let $N$ be the nilradical of $R$.  Then there is some power $q$ of $p$ such that $N^{[q]} = 0$. 

Now let $\beta \in \ia^{[p]} \cap (\alpha^p)$.  Let $y\in R$ with residue class equal to $\beta$.  Then $y \in (I^{[p]} + N) \cap ((x^p) + N)$.  Then \[
y^q \in (I^{[p]} +N)^{[q]} \cap ((x^p)+N)^{[q]} = I^{[pq]} \cap (x^{pq}) = (I \cap (x))^{[pq]},
\]
where the last equality follows from (1) and induction on $\log_pq$.

It follows that $\beta^q \in (\ia \cap (\alpha))^{[pq]} = ((\ia \cap (\alpha))^{[p]})^{[q]}$, so that $\beta \in ((\ia \cap (\alpha))^{[p]})^F$.  But every ideal in $R_{red}$ is Frobenius closed, so $\beta \in (\ia \cap (\alpha))^{[p]}$.  Thus, $\ia^{[p]} \cap (\alpha^p) \subseteq (\ia \cap (\alpha))^{[p]}$, whence equality holds. Then by the Main Theorem, $R_{red}$ is regular. 
\end{proof}

However, the converse is false.  To see this, let $R=k[x,y,z,w] / (x^{p+1}, x^p z - y^p w, x y, y^{p+1})$.  Then $(x) \cap (y) = (x^p z)$, so $((x) \cap (y))^{[p]} = 0$, but $x^p z \in (x^p) \cap (y^p)$.  On the other hand, it is well known that the Frobenius on $R_{red} \cong k[z,w]$ is F-intersection flat, hence Ohm-Rush.  Thus, the Ohm-Rush property of the Frobenius on nonreduced rings remains mysterious.

\section*{Acknowledgment}
The author thanks Craig Huneke for comments on the proof of Lemma~\ref{lem:zd} and Wenliang Zhang for making the author aware of Zhang's previous contribution.

\providecommand{\bysame}{\leavevmode\hbox to3em{\hrulefill}\thinspace}
\providecommand{\MR}{\relax\ifhmode\unskip\space\fi MR }
\providecommand{\MRhref}[2]{%
  \href{http://www.ams.org/mathscinet-getitem?mr=#1}{#2}
}
\providecommand{\href}[2]{#2}


\begin{thebibliography}{Kun69}

\bibitem[Bou72]{Bour-CA}
Nicolas Bourbaki, \emph{Elements of mathematics, {C}ommutative algebra},
  Hermann, Paris, 1972, Translated from the French (1961--1965).

\bibitem[Che43]{Che-local}
Claude Chevalley, \emph{On the theory of local rings}, Ann. of Math. (2)
  \textbf{44} (1943), 690--708.

\bibitem[Eps21]{nme-exctest}
Neil Epstein, \emph{Test elements, excellent rings, and content functions}, in
  preparation, 2021.

\bibitem[ES16]{nmeSh-OR}
Neil Epstein and Jay Shapiro, \emph{The {O}hm-{R}ush content function}, J.
  Algebra Appl. \textbf{15} (2016), no.~1, 1650009, 14 pp.

\bibitem[EY12]{nmeYao-flat}
Neil Epstein and Yongwei Yao, \emph{Criteria for flatness and injectivity},
  Math. Z. \textbf{271} (2012), no.~3-4, 1193--1210.

\bibitem[HH94]{HHbase}
Melvin Hochster and Craig Huneke, \emph{{$F$}-regularity, test elements, and
  smooth base change}, Trans. Amer. Math. Soc. \textbf{346} (1994), no.~1,
  1--62.

\bibitem[HJ20]{HoJe-flatness}
Melvin Hochster and Jack Jeffries, \emph{Extensions of primes, flatness, and
  intersection flatness}, {arXiv}:2003.02560 [math.AC], 2020.

\bibitem[Kat08]{Ka-param}
Mordechai Katzman, \emph{Parameter-test-ideals of {C}ohen-{M}acaulay rings},
  Comp. Math. \textbf{144} (2008), no.~4, 933--948.

\bibitem[Kun69]{Kunz-regp}
Ernst Kunz, \emph{Characterizations of regular local rings for characteristic
  {$p$}}, Amer. J. Math. \textbf{91} (1969), 772--784.

\bibitem[Mil03]{Mil-Frobhd}
Claudia Miller, \emph{The {F}robenius endomorphism and homological dimensions},
  Commutative algebra ({G}renoble/{L}yon, 2001), Contemp. Math., vol. 331,
  Amer. Math. Soc., Providence, RI, 2003, pp.~207--234.

\bibitem[OR72]{OhmRu-content}
Jack Ohm and David~E. Rush, \emph{Content modules and algebras}, Math. Scand.
  \textbf{31} (1972), 49--68.

\bibitem[Sha12]{Sha-nonred}
Rodney~Y. Sharp, \emph{Big tight closure test elements for some non-reduced
  excellent rings}, J. Algebra \textbf{349} (2012), 284--316.

\bibitem[Zha09]{Zh-Frobcolon}
Wenliang Zhang, \emph{On the {F}robenius power and colon ideals}, Comm. Algebra
  \textbf{37} (2009), no.~7, 2391--2395.

\end{thebibliography}
\end{document}